\documentclass[10pt]{amsart}

\usepackage[margin=1.25in]{geometry} 
\usepackage{amsmath,amsthm,amssymb,bbm,bm}
\usepackage{graphicx}
\usepackage{hyperref}
\usepackage{amssymb,amsfonts, bm}
\usepackage[all,arc]{xy}
\usepackage{enumerate}
\usepackage{mathrsfs}
\usepackage{tikz}
\usepackage{caption}
\usepackage{subcaption}

\usetikzlibrary{shapes.geometric}

\newcommand{\R}{\mathbb{R}}

\newcommand{\Gc}{\mathcal{G}}

\newcommand{\I}{\mathbbm{1}}
\newcommand{\Var}{\mathrm{Var}}
\newcommand{\cov}{\mathrm{cov}}

\newcommand{\Ber}{\mathtt{Bernoulli}}
\newcommand{\CM}{\mathrm{CM}}

\usepackage[backend=biber,style=numeric,hyperref,maxbibnames=99]{biblatex}
\renewbibmacro{in:}{}

\newcommand{\Address}{{
\bigskip
\footnotesize
\textsc{Institute for Mathematical Stochastics, University of G\"ottingen, G\"ottingen, 37077, Germany }\par\nopagebreak
\textit{E-mail address}: \texttt{santiago.arenasvelilla@uni-goettingen.de} \\
  
\textsc{Centro de Investigaci\'{o}n en Matem\'{a}ticas, Guanajuato, Gto. 36000, Mexico}\par\nopagebreak
\textit{E-mail address}: \texttt{octavius@cimat.mx} \\

\textsc{Department of Mathematics \& Statistics, McMaster University, Hamilton, ON, L8S 4K1, Canada}\par\nopagebreak
\textit{E-mail address}: \texttt{paguyoj@mcmaster.ca}
}}

\def\bal#1\eal{\begin{align*}#1\end{align*}}

\newtheorem{theorem}{Theorem}[section]
\newtheorem{lemma}[theorem]{Lemma}
\newtheorem{proposition}[theorem]{Proposition}

\newtheorem{rem}[theorem]{Remark}

\title{Central limit theorem for crossings in randomly embedded graphs}
\author{Santiago Arenas-Velilla, Octavio Arizmendi, and J. E. Paguyo} 
\date{}

\subjclass[2020]{60C05, 60F05}
\keywords{random graphs, crossings, size-bias coupling, Stein's method, central limit theorem, convergence rates}

\begin{document}


\begin{abstract}
We consider the number of crossings in a random embedding of a graph, $G$, with vertices in convex position. 
We give explicit formulas for the mean and variance of the number of crossings as a function of various subgraph counts of $G$. 
Using Stein's method and size-bias coupling, we establish an upper bound on the Kolmogorov distance between the distribution of the number of crossings and a standard normal random variable. We also consider the case where $G$ is a random graph and obtain a Kolmogorov bound between the distribution of crossings and a Gaussian mixture distribution.
As applications, we obtain central limit theorems with convergence rates 
for the number of crossings in random embeddings of matchings, path graphs, cycle graphs, disjoint union of triangles, random $d$-regular graphs, and mixtures of random graphs. 
\end{abstract}

\maketitle


\section{Introduction}

Let $G = (V,E)$ be a graph with vertex set $V = [n] := \{1,\ldots, n\}$ and edge set $E$. Let $\pi \in S_n$ be a permutation on $[n]$. 
An {\em embedding} of $G$ is the graph, $G_\pi$, obtained from the graph isomorphism induced by $\pi$, so that $(v,w) \in E$ is an edge in $G$ if and only if $(\pi(v), \pi(w))$ is an edge in $G_\pi$. 
A {\em crossing} in an embedding $G_\pi$ is a quadruple $(a,b,c,d)$ such that $(a,b), (c,d) \in E$ and $\pi(a) > \pi(c) > \pi(b) > \pi(d)$. 
For example, Figure \ref{pathfig_example} shows two possible embeddings of the path graph $P_{20}$. The first embedding has $40$ crossings, while the second has $60$ crossings. 

Let $\pi \in S_n$ be a uniformly random permutation. A {\em random embedding} of $G$ is the random graph, $G_\pi$, obtained from the graph isomorphism induced by the random permutation $\pi$. 
We also refer to the random graph obtained through this random embedding as a {\em randomly embedded graph}. 

In this paper, we study the distribution of the number of crossings, $X$, in a random embedding of the graph $G$, where the vertices are assumed to be in convex position, or on a circle, for practical purposes. We compute explicit formulas for the mean and variance of $X$ in terms of subgraphs of certain types. 
We prove an upper bound on the Kolmogorov distance between the distribution of $X$ and a standard normal random variable. 
We apply this result to show the asymptotic normality, along with convergence rates, of the number of crossings in several families of random graphs.
 
\begin{figure}[ht]
\centering
\begin{subfigure}[b]{0.45\linewidth}
\includegraphics[width=150pt, height=150pt]{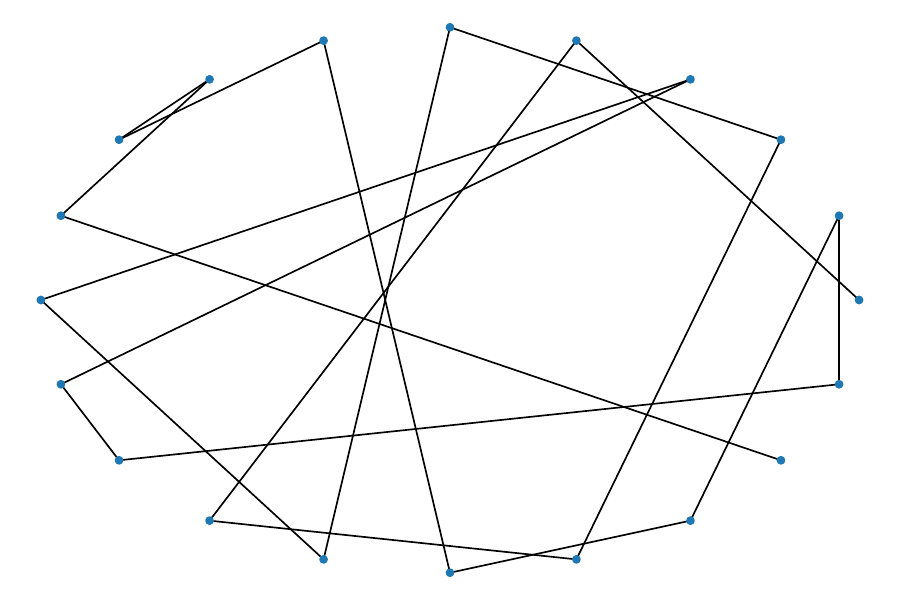}
\caption{$P_{20}$ with 40 crossings}
\label{P_297}
\end{subfigure}
\begin{subfigure}[b]{0.45\linewidth}
\includegraphics[width=150pt, height=150pt]{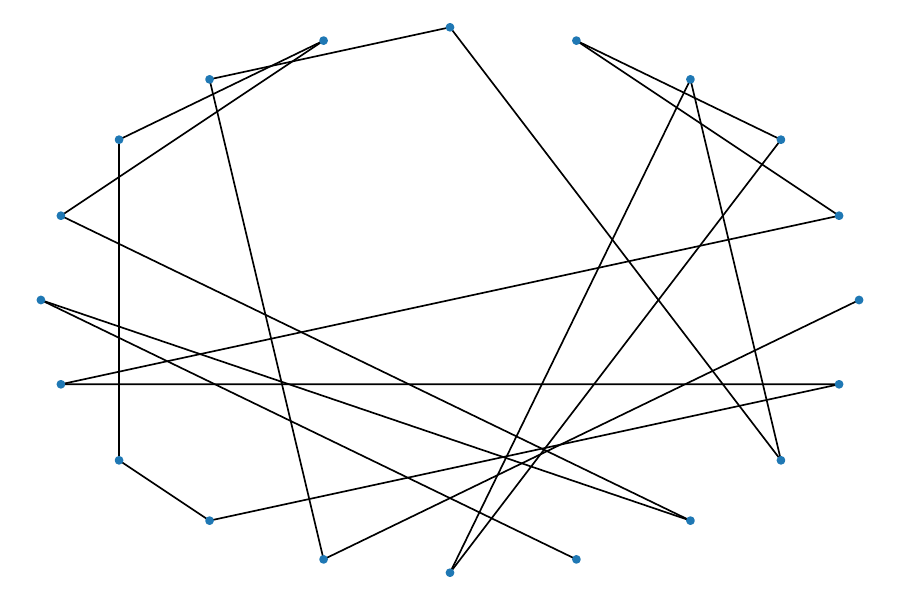}
\caption{$P_{20}$ with 60 crossings}
\label{P_383}
\end{subfigure}
\caption{Two possible embeddings of the path graph on $20$ vertices.}
\label{pathfig_example}
\end{figure}

There has been some previous work on the number of crossings in random graphs. 
One of the earliest results is due to Moon \cite{Moo65}, who showed that the number of crossings in a complete graph embedded randomly on a sphere is asymptotically normal. 
Riordan \cite{Rio75} considered random {\em chord diagrams}, graphs obtained by pairing $2n$ points on a circle, where each pairing corresponds to a chord, 
and found an explicit formula in the form of an alternating sum for the number of chord diagrams with exactly $k$ crossings. 
Around twenty years later, Flajolet and Noy \cite{FN00} used analytic combinatorics to prove the asymptotic normality of the number of crossings in a random chord diagram. 
This result was also recently proven by Feray as an application of his theory of weighted dependency graphs \cite{Fer18}. 
Using cumulants, Arizmendi, Cano, and Huemer \cite{ACH19} showed that the number of crossings in a random labeled tree is asymptotically normal, 
and Arenas-Velilla and Arizmendi \cite{AA23} gave a quantitative version. 
More recently, Chern, Diaconis, Kane, and Rhoades \cite{CDKR15} established the asymptotic normality of the number of crossings in a uniformly random set partition of $[n]$. 
Subsequently, Feray \cite{Fer20} generalized this result and used his theory of weighted dependency graphs to prove that the number of occurrences of 
any fixed pattern in random multiset permutations and random set partitions is asymptotically normal.

One motivation for the study of crossings in random graphs is the notion of the {\em crossing number} of a graph $G$, which is the smallest number of edge crossings in any plane drawing of $G$. 
The study of crossing numbers originated from Turan's brick factory problem \cite{Tur77}, which asked for the factory plan which minimized the number of crossings between tracks connecting brick kilns and storage sites. 
This is equivalent to determining the crossing number of a complete bipartite graph. The crossing number has been studied for various graphs and graph families. 
For example, Spencer and Toth \cite{ST02} initiated the study of the crossing number of random graphs and proved several results on its expected value. 
However, the crossing number remains unknown except for a few special cases. In fact, Garey and Johnson \cite{GJ83} showed that determining the crossing number of a graph is an NP-hard problem. 
We refer the reader to Schaefer's survey \cite{Sch13} for the history, previous results, and variants on the crossing number. 


\subsection{Main Result}

Let $\mu$ and $\nu$ be probability distributions. The {\em Kolmogorov distance} between $\mu$ and $\nu$ is 
\bal
d_K(\mu, \nu) := \sup_{x\in \R} |\mu(-\infty, x] - \nu(-\infty, x]|.
\eal
If $X$ and $Y$ are random variables with cumulative distribution functions $F_X$ and $F_Y$, respectively, then we write $d_K(F_X,F_Y)$ to denote the Kolmogorov distance between the distributions of $X$ and $Y$. We will also abuse notation and define $d_K(X,Y) := d_K(F_X,F_Y)$, and we use this notation when it is more convenient. 

Our main result is an upper bound on the Kolmogorov distance between the distribution of the number of crossings in a random embedding of $G$ and a  normal random variable. 
The proof uses Stein's method with size-bias coupling. 

\begin{theorem} \label{maintheorem}
Let $G$ be a graph on $m$ edges with maximum degree $\Delta$. Let $X$ be the number of crossings in a uniformly random embedding of $G$. Let $W = \frac{X - \mu}{\sigma}$, where $\mu = E(X)$, $\sigma^2 = \Var(X)$, and let $Z$ be a standard normal random variable. Then 
\begin{enumerate}
\item for $\Delta \in \{1,2\}$,
\bal
d_K(W, Z) \leq  \frac{1}{\sqrt{m}} \left(60(1+9\sqrt{5})\Delta^2 + 600 \Delta \right),
\eal
\item for $\Delta \geq 3$, 
\bal
d_K(W, Z) \leq  \frac{1}{\sqrt{m}} \left(120(1+9\sqrt{10})\Delta^2 + 1200 \Delta \right).
\eal
\end{enumerate}
\end{theorem}

Observe that this implies a central limit theorem, along with a rate of convergence, for the number of crossings in a random embedding of any graph with bounded constant degree, or with maximum degree $\Delta = o(m^{1/4})$. We also remark that the mean and variance of $X$ can be explicitly calculated in terms of very precise subgraph counts of $G$.

This theorem can also be applied to random graphs. In this setting, we obtain a mixture distribution and approximate it using a mixture of Gaussian distributions. Recall that for a countable set of cumulative distribution functions, $\{F_i\}_{i=1}^\infty$, and a countable set of nonnegative weights $\{p_i\}_{i=1}^\infty$ such that $\sum_i p_i=1$, the cumulative distribution function
$F = \sum_{i=1}^\infty p_i F_i$ is called a {\em mixture} or a {\em mixture distribution}. If $F_i$ is a Gaussian distribution for all $i$, then $F$ is a {\em Gaussian mixture distribution}.

\begin{theorem} \label{maintheoremrandom}
Let $\{G_i\}_{i \in I}$ be a sequence of graphs on $n$ vertices and let $\Gc$ be a random graph such that $p_i = P(\Gc = G_i)$, where $\sum_{i \in I} p_i = 1$ and $I$ is a countable index set. Let $m(G_i)$ be the number of edges and $\Delta(G_i)$ the maximum degree of $G_i$, and suppose there exists $\Delta$ and $m$ such that $m(G_i)\geq m$ and $\Delta(G_i) \leq \Delta$ for all $i \in I$. Let $X := X(\Gc)$ be the number of crossings in a uniformly random embedding of $\Gc$. Let $\mu_i = E(X |\Gc = G_i)$ and $\sigma_i^2 = \Var(X|\Gc=G_i)$. Let $Z_i$ be a normal random variable with mean $\mu_i$ and variance $\sigma_i^2$, for all $i \in I$. Then for all $m$,
\bal
d_K\left(F_X, \sum_{i \in I} p_i F_{Z_i} \right) \leq \frac{C}{\sqrt{m}}
\eal
for some constant $C = C(\Delta)$ possibly depending on $\Delta$.
\end{theorem}


\subsection{Outline}

The paper is organized as follows. 
Section \ref{Preliminaries} introduces the notation and definitions that are used throughout the paper, and gives necessary background and results on size-bias coupling and Stein’s method.

In Section \ref{MeanVariance}, we provide explicit formulas for the mean and variance of the the number of crossings in a random embedding. 
In Section \ref{CLTCrossings}, we construct a size-bias coupling and use this with Stein's method to prove Theorem \ref{maintheorem}, the central limit theorem for the number of crossings. 

In Section \ref{Examples}, we apply Theorem \ref{maintheorem} to prove central limit theorems, with convergence rates, for the number of crossings in matchings, path graphs, cycle graphs, 
and the disjoint union of $n$ triangles. 
In Section \ref{RandomExamples} we consider crossings in random embeddings of random graphs and prove Theorem \ref{maintheoremrandom}, then apply it to some examples of random graphs. In particular we prove the asymptotic normality for the number of crossings of a $d$-regular graph. 

Section \ref{AnotherLimit} discusses another possible limit that can arise for certain graphs. We conclude the paper with some final remarks and open questions in Section \ref{FinalRemarks}.


\section{Preliminaries} \label{Preliminaries} 

In this section, we establish the notation and definitions that we use throughout the paper. We also discuss size-bias coupling and Stein's method, 
the main tool that we use to prove our central limit theorem. 

\subsection{Notation and Definitions on Graphs} \label{notationanddefinitions}

A {\em graph} is a pair $G = (V,E)$, where $V$ is the set of vertices and $E \subseteq \{ (v,w) : v,w \in V\}$ is the set of edges. 
We sometimes write the edge $(v,w)$ as $v \sim_G w $, or as $v \sim w$ if the underlying graph $G$ is clear. We let $n = |V|$ denote the number of vertices and $m = |E|$ the number of edges.

Fix a vertex $v$. Vertex $w$ is a {\em neighbor} of $v$, if $(v,w)\in E$. The number of neighbors of $v$ is the {\em degree} of $v$, denoted by $\deg(v)$. 
The {\em maximum degree} of a graph $G$ is defined as $\Delta := \Delta(G) = \max_{v \in V} \deg(v)$. 

A {\em subgraph} of $G$ is a graph $H=(W,F)$, such that $W \subseteq V$ and $F \subseteq G$. 
Two graphs $G = (V,E)$ and $G' = (V',E')$ are {\em isomorphic} if there exist a bijection $\varphi : V \to V'$ such that $(u , v) \in E $ if and only if $(\varphi(u) , \varphi(v)) \in E'$, 
for all $u,v \in V$. 

An \textit{$r$-matching} in a graph $G$ is a set of $r$ edges in $G$, no two of which have a vertex in common. 
We let $M_r(G)$ be the set $r$-matchings of $G$ and $m_r(G) := |M_r(G)|$ the number of $r$-matchings of $G$. Note that $m_1(G) = m$ corresponds to the number of edges of $G$. 

\subsection{Size-Bias Coupling and Stein's Method} \label{sizebiascouplingSteinmethod}

Let $X$ be a non-negative random variable with finite non-zero mean $\mu$. Then $X^s$ has the {\em size-bias distribution} of $X$ if for all $f$ such that $E(Xf(X)) < \infty$, we have
\bal
E(Xf(X)) = \mu E(f(X^s)). 
\eal
A {\em size-bias coupling} is a pair $(X, X^s)$ of random variables defined on the same probability space such that $X^s$ has the size-bias distribution of $X$. 
To prove our main theorem, we will use the following size-bias coupling version of Stein's method. 

\begin{theorem}[\cite{CGS11}, Theorem 5.6] \label{sizebiasStein}
Let $X \geq 0$ be a random variable with finite non-zero  mean $\mu$ and positive, finite variance $\sigma^2$. Suppose $(X, X^s)$ is a bounded size-bias coupling, so that $|X^s - X| \leq A$ for some $A$. 
If $W = \frac{X - \mu}{\sigma}$ and $Z$ is a standard normal random variable, then 
\bal
d_K(W,Z) \leq \frac{6 \mu A^2}{\sigma^3} + \frac{2\mu}{\sigma^2} \sqrt{\Var(E[X^s - X \mid X])}.
\eal
\end{theorem}

Observe that this implies a central limit theorem if both error terms on the right hand side go to zero. 

\subsection{Size-Bias Coupling Construction} \label{CouplingConstruction}

We now turn to a method of coupling a random variable $X$ to a size-biased version $X^s$. In the case that $X = \sum_{i=1}^n X_i$, where $X_i \geq 0$ and $E(X_i) = \mu_i$, 
we can use the following recipe provided in \cite{Ros11} to construct a size-biased version of $X$. 

\begin{enumerate}
\item For each $i \in [n]$, let $X_i^s$ have the size-bias distribution of $X_i$ independent of $(X_k)_{k \neq i}$ and $(X_k^s)_{k \neq i}$. 
Given $X_i^s = x$, define the vector $(X_k^{(i)})_{k \neq i}$ to have the distribution of $(X_k)_{k \neq i}$ conditional on $X_i = x$. 
\item Choose a random summand $X_I$, where the index $I$ is chosen, independent of all else, with probability $P(I = i) = \mu_i/\mu$, where $\mu = EX$. 
\item Define $X^s = \sum_{k \neq I} X_k^{(I)} + X_I^s$. 
\end{enumerate}
If $X^s$ is constructed by Items (1)-(3) above, then $X^s$ has the size-bias distribution of $X$. 

As a special case, note that if $X_i$ is an indicator random variable, then its size-bias distribution is $X_i^s = 1$. 
We summarize this as the following proposition. 

\begin{proposition}[\cite{Ros11}, Corollary 3.24]\label{sizebiascouplingrecipe}
Let $X_1,\ldots, X_n$ be zero-one random variables and let $p_i := P(X_i = 1)$. For each $i \in [n]$, let $(X_k^{(i)})$ have the distribution of $(X_k)_{k \neq i}$ conditional on $X_i = 1$. 
If $X = \sum_{i=1}^n X_i$, $\mu = EX$, and $I$ is chosen independent of all else with $P(I = i) = p_i/\mu$, then $X^s = \sum_{k \neq I} X_k^{(I)} + 1$ has the size-bias distribution of $X$. 
\end{proposition}


\section{Mean and Variance} \label{MeanVariance}

Fix a uniformly random permutation $\pi \in S_n$ and let $X := X(G_\pi)$ be the number of crossings in the corresponding randomly embedded graph $G_\pi$. 
Then we can write $X$ as a sum of indicator random variables 
\bal
X = \sum_{k \in M_2(G)} Y_k,
\eal
where $M_2(G)$ is the set of 2-matchings of $G$ and $Y_k$ is the indicator random variable that the pair of edges indexed by $k$ forms a crossing in $G$. 

In this section we give a formula for the mean and variance of the random variable $X$ in terms of the number of subgraphs of certain types.

\begin{theorem}\label{crossingsmeanvariance}
Let $G$ be a graph on $n$ vertices and let $X$ be the number of crossings in a uniformly random embedding of $G$.
\bal
\mu = E(X) &= \frac{1}{3}m_2(G), \\
\sigma^2 = \Var(X) &= \frac{2}{15}m_3(G) + \frac{2}{9}m_2(G)+ \frac{1}{45}S_4 - \frac{1}{45}S_5 - \frac{1}{18}S_6 + \frac{1}{9}S_7,
\eal
where $m_r(G)$ is the number of r-matchings of $G$ and $S_i$ is the number of subgraphs of $G$ of type $C_i$, as defined in Figure \ref{secondmomentcrossingconfigs}.
\end{theorem}

\begin{proof}
For all $k\in M_2(G)$, observe that $Y_k \sim \Ber(1/3)$. Indeed, if $k$ consists of the two edges $e_1 = (v_1,v_2)$  and $e_2 = (v_3,v_4)$, 
then the probability of a crossing depends only on the cyclic order in which $v_1,v_2,v_3,v_4$ are embedded in $\{1,\dots,n\}$, not on the precise positions of them. 
From the six possible cyclic orders, only $1/3$ yield a crossing; see Figure \ref{six_orders}. 

Therefore $E(Y_k) = 1/3$, so by linearity of expectation we get
\bal
E(X) &= \sum_{k \in M_2(G)} E(Y_k) = \frac{1}{3} m_2(G).
\eal
For the second moment, we expand the sum to get
\bal
E(X^2) = \sum_{k,\ell \in M_2(G)} E(Y_k Y_\ell) = \sum_{k,\ell \in M_2(G)} P(k,\ell \text{ are crossings in } G).
\eal
Note that the probability $P(k,\ell \text{ are crossings in } G)$ depends on the number of edges and vertices that the two $2$-matchings, $k$ and $\ell$, share. 
Thus the sum above can be divided into eight different types; see Figure \ref{secondmomentcrossingconfigs}. 

Let $A_i$ be the event that $k$ and $\ell$ are both crossings in $G$ of type $C_i$. Then
\bal
\begin{tabular}{c c c c}
$P(A_1) = 1/9$, & $P(A_2)  = 1/9$ , & $P(A_3) = 2/15$, & $P(A_4) = 7/60$, \\
$P(A_5) = 1/10$ , & $P(A_6) = 1/12$, & $P(A_7) = 1/6$, & $P(A_8) = 1/3$.
\end{tabular}
\eal 
Summing over all types yields
\bal
E(X^2) = \frac{2}{3}m_4(G)+\frac{4}{5}m_3(G)+\frac{1}{3}m_2(G)+\frac{4}{9}S_2+ \frac{7}{15}S_4+\frac{1}{5}S_5+\frac{1}{6}S_6+ \frac{1}{3}S_7,    
\eal
where $S_i$ is the number of subgraphs of $G$ of type $C_i$.

Next, we can rewrite $E(X)^2$ by expanding the sum similar to the calculation of $E(X^2)$. Indeed, 
\bal
E(X)^2 &= \sum_{k,\ell \in M_2(G)} E(Y_k)E( Y_\ell) \\ &= \sum_{k,\ell \in M_2(G)} P(k\text{ is a crossing in } G)P(\ell\text{ is a crossing in } G).
\eal
In this case, $P(k\text{ is a crossing in } G)P(\ell\text{ is a crossing in } G)=1/9$, so again by considering different possibilities for the intersection of $k$ and $l$, we obtain 
\bal
E(X)^2= \frac{2}{3}m_4(G) + \frac{2}{3}m_3(G) + \frac{1}{9}m_2(G)  +\frac{4}{9}S_2 + \frac{4}{9}S_4 + \frac{2}{9}S_5 + \frac{2}{9}S_6 + \frac{2}{9}S_7.
\eal
The expression for the variance follows.
\end{proof}

\begin{figure}
\centering 
\begin{tikzpicture}[scale = 0.75]

\draw (0,0) node[black][left] {$v_1$} -- (1,0) ;
\filldraw (0,0) circle (1.5pt);
\filldraw (1,0) node[black][right]{$v_2$} circle (1.5pt);
\draw (0,-1) node[black][left] {$v_4$} -- (1,-1) ;
\filldraw (0,-1) circle (1.5pt);
\filldraw (1,-1) node[black][right]{$v_3$} circle (1.5pt);

\draw (3,0) node[black][left] {$v_1$} -- (4,0) ;
\filldraw (3,0) circle (1.5pt);
\filldraw (4,0) node[black][right]{$v_2$} circle (1.5pt);
\draw (3,-1) node[black][left] {$v_3$} -- (4,-1) ;
\filldraw (3,-1) circle (1.5pt);
\filldraw (4,-1) node[black][right]{$v_4$} circle (1.5pt); 

\draw (6,0) node[black][left] {$v_1$} -- (7,-1) ;
\filldraw (6,0) circle (1.5pt);
\filldraw (7,-1) node[black][right]{$v_2$} circle (1.5pt);
\draw (6,-1) node[black][left] {$v_4$} -- (7,0) ;
\filldraw (6,-1) circle (1.5pt);
\filldraw (7,0) node[black][right]{$v_3$} circle (1.5pt);

\draw (9,0) node[black][left] {$v_1$} -- (10,-1) ;
\filldraw (9,0) circle (1.5pt);
\filldraw (10,-1) node[black][right]{$v_2$} circle (1.5pt);
\draw (9,-1) node[black][left] {$v_3$} -- (10,0) ;
\filldraw (9,-1) circle (1.5pt);
\filldraw (10,0) node[black][right]{$v_4$} circle (1.5pt);

\draw (12,0) node[black][left] {$v_1$} -- (12,-1) ;
\filldraw (12,0) circle (1.5pt);
\filldraw (12,-1) node[black][left]{$v_2$} circle (1.5pt);
\draw (13,0) node[black][right] {$v_3$} -- (13,-1) ;
\filldraw (13,0) circle (1.5pt);
\filldraw (13,-1) node[black][right]{$v_4$} circle (1.5pt);

\draw (15,0) node[black][left] {$v_1$} -- (15,-1) ;
\filldraw (15,0) circle (1.5pt);
\filldraw (15,-1) node[black][left]{$v_2$} circle (1.5pt);
\draw (16,0) node[black][right] {$v_4$} -- (16,-1) ;
\filldraw (16,0) circle (1.5pt);
\filldraw (16,-1) node[black][right]{$v_3$} circle (1.5pt);

\end{tikzpicture}
\caption{The six possible cyclic orders for a $2$-matching.}
\label{six_orders}
\end{figure}
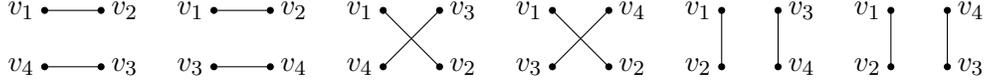

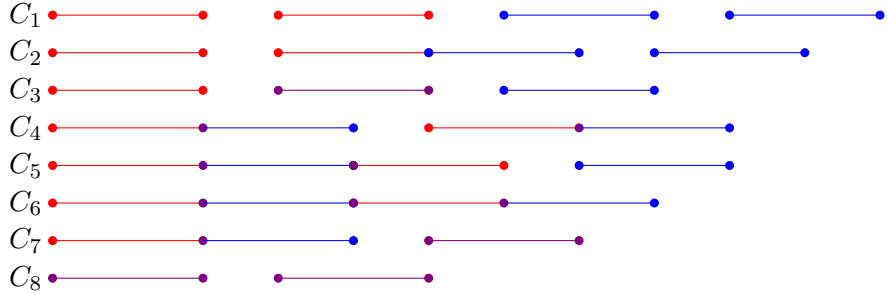
\begin{figure}
\centering 
\begin{tikzpicture}[scale = 1]
\draw [red] (0,0) node[black][left] {$C_1$} -- (2,0) ;
\filldraw [red] (0,0) circle (1.5pt);
\filldraw [red] (2,0) circle (1.5pt);
\draw [red] (3,0) -- (5,0);
\filldraw [red] (3,0) circle (1.5pt);
\filldraw [red] (5,0) circle (1.5pt);
\draw [blue](6,0) -- (8,0);
\filldraw [blue] (6,0) circle (1.5pt);
\filldraw [blue] (8,0) circle (1.5pt);
\draw [blue](9,0) -- (11,0);
\filldraw [blue] (9,0) circle (1.5pt);
\filldraw [blue] (11,0) circle (1.5pt);

\draw [red] (0,-0.5) node[black][left] {$C_2$} -- (2,-0.5);
\filldraw [red] (0,-0.5) circle (1.5pt);
\filldraw [red] (2,-0.5) circle (1.5pt);
\draw [red] (3,-0.5) -- (5,-0.5);
\filldraw [red] (3,-0.5) circle (1.5pt);
\filldraw [red] (5,-0.5) circle (1.5pt);
\draw [blue](5,-0.5) -- (7,-0.5);
\filldraw [blue] (5,-0.5) circle (1.5pt);
\filldraw [blue] (7,-0.5) circle (1.5pt);
\draw [blue](8,-0.5) -- (10,-0.5);
\filldraw [blue] (8,-0.5) circle (1.5pt);
\filldraw [blue] (10,-0.5) circle (1.5pt);

\draw [red] (0,-1) node[black][left] {$C_3$} -- (2,-1);
\filldraw [red] (0,-1) circle (1.5pt);
\filldraw [red] (2,-1) circle (1.5pt);
\draw [violet] (3,-1) -- (5,-1);
\filldraw [violet] (3,-1) circle (1.5pt);
\filldraw [violet] (5,-1) circle (1.5pt);
\draw [blue](6,-1) -- (8,-1);
\filldraw [blue] (6,-1) circle (1.5pt);
\filldraw [blue] (8,-1) circle (1.5pt);

\draw [red] (0,-1.5) node[black][left] {$C_4$} -- (2,-1.5);
\filldraw [red] (0,-1.5) circle (1.5pt);
\filldraw [violet] (2,-1.5) circle (1.5pt);
\draw [blue](2,-1.5) -- (4,-1.5);
\filldraw [blue] (4,-1.5) circle (1.5pt);
\draw [red] (5,-1.5) -- (7,-1.5);
\filldraw [red] (5,-1.5) circle (1.5pt);
\filldraw [violet] (7,-1.5) circle (1.5pt);
\draw [blue](7,-1.5) -- (9,-1.5);
\filldraw [blue] (9,-1.5) circle (1.5pt);

\draw [red] (0,-2) node[black][left] {$C_5$}-- (2,-2);
\filldraw [red] (0,-2) circle (1.5pt);
\filldraw [violet] (2,-2) circle (1.5pt);
\draw [blue](2,-2) -- (4,-2);
\filldraw [blue] (4,-2) circle (1.5pt);
\draw [red] (4,-2) -- (6,-2);
\filldraw [red] (6,-2) circle (1.5pt);
\filldraw [violet] (4,-2) circle (1.5pt);
\draw [blue](7,-2) -- (9,-2);
\filldraw [blue] (9,-2) circle (1.5pt);
\filldraw [blue] (7,-2) circle (1.5pt);

\draw [red] (0,-2.5) node[black][left] {$C_6$}-- (2,-2.5);
\filldraw [red] (0,-2.5) circle (1.5pt);
\filldraw [violet] (2,-2.5) circle (1.5pt);
\draw [blue](2,-2.5) -- (4,-2.5);
\filldraw [blue] (4,-2.5) circle (1.5pt);
\draw [red] (4,-2.5) -- (6,-2.5);
\filldraw [violet] (6,-2.5) circle (1.5pt);
\filldraw [violet] (4,-2.5) circle (1.5pt);
\draw [blue](6,-2.5) -- (8,-2.5);
\filldraw [blue] (8,-2.5) circle (1.5pt);

\draw [red] (0,-3) node[black][left] {$C_7$} -- (2,-3);
\filldraw [red] (0,-3) circle (1.5pt);
\filldraw [violet] (2,-3) circle (1.5pt);
\draw [blue](2,-3) -- (4,-3);
\filldraw [blue] (4,-3) circle (1.5pt);
\draw [violet] (5,-3) -- (7,-3);
\filldraw [violet] (5,-3) circle (1.5pt);
\filldraw [violet] (7,-3) circle (1.5pt);

\draw [violet] (0,-3.5) node[black][left] {$C_8$} -- (2,-3.5);
\filldraw [violet] (0,-3.5) circle (1.5pt);
\filldraw [violet] (2,-3.5) circle (1.5pt);
\draw [violet] (3,-3.5) -- (5,-3.5);
\filldraw [violet] (3,-3.5) circle (1.5pt);
\filldraw [violet] (5,-3.5) circle (1.5pt);

\end{tikzpicture}
\caption{The eight types of $2$-matchings which appear in the sum of the second moment of $X$.}
\label{secondmomentcrossingconfigs}
\end{figure}


\section{Central Limit Theorem for Crossings} \label{CLTCrossings}

In this section, we construct a size-bias coupling for the number of crossings using the recipe outlined in Section \ref{sizebiascouplingrecipe}. 
We then apply the size-bias coupling version of Stein’s method, Theorem \ref{sizebiasStein}, to prove Theorem \ref{maintheorem}. 
The main difficulty is in bounding the conditional variance term of Theorem \ref{sizebiasStein}.

\subsection{Construction of the Size-Bias Coupling} \label{sizebiascouplingconstruction}

Following the coupling recipe from Section \ref{sizebiascouplingrecipe}, we construct a random variable $X^s$ having the size-bias distribution with respect to the number of crossings $X$.

Fix a permutation $\pi \in S_n$ and let $G_\pi$ be the corresponding embedded graph. Let $X$ be the number of crossings in $G_\pi$. 
Pick an index $I = (e,f) \in M_2(G)$ uniformly at random from the set of $2$-matchings of $G$, independent of all else. In particular, this is independent of $\pi$. 
Suppose the edges $e,f$ are given by $e = (u_1, u_2)$ and $f = (v_1, v_2)$. 
Under the embedding defined by $\pi$, the corresponding edges $e', f'$ in $G_\pi$ are given by $e' = (\pi(u_1), \pi(u_2))$ and $f' = (\pi(v_1), \pi(v_2))$. 

We construct a new permutation $\pi^s$ from $\pi$ as follows. 
If there is already a crossing in $G_\pi$ at index I, set $\pi^s = \pi$. 
Otherwise, edges $e', f'$ do not cross in $G_\pi$. Without loss of generality, we may assume that $\pi(u_1) < \pi(v_1) < \pi(v_2) < \pi(u_2)$. 
Choose a vertex among $u_1, u_2, v_1, v_2$ uniformly at random. 
If vertex $u_1$ or $v_1$ is chosen, define $\pi^s$ such that $\pi^s(u_2) = \pi(v_2), \pi^s(v_2) = \pi(u_2)$, and $\pi^s = \pi$ otherwise.  
If vertex $u_2$ or $v_2$ is chosen, set $\pi^s(u_1) = \pi(v_1), \pi^s(v_1) = \pi(u_1)$, and $\pi^s = \pi$ otherwise. 
Finally, let $G_{\pi^s}$ be the corresponding embedded graph and note that the edges $e'' = (\pi^s(u_1), \pi^s(u_2))$ and $f'' = (\pi^s(v_1), \pi^s(v_2))$ form a crossing in $G_{\pi^s}$. 
That is, $G_{\pi^s}$ has a crossing at index $I$. 

Set $X^s := \sum_{k \neq I} Y_k^{(I)} + 1$ where $Y_k^{(I)}$ is the indicator random variable that there is a crossing in $G_{\pi^s}$ at index $k$. 
Observe that by construction, $\pi^s$ is a uniformly random permutation conditioned on the event that $\pi(u_1), \pi(u_2), \pi (v_1), \pi(v_2)$ are in alternating cyclic order. 
It follows that $G_{\pi^s}$ is a uniformly random embedding of $G$ conditioned to have a crossing at index $I$. 
Thus $(Y^{(I)})$ has the marginal distribution of $(Y_k^{(I)})_{k \neq I}$ conditional on $Y_I = 1$. 
By Proposition \ref{sizebiascouplingrecipe}, $X^s$ has the size-bias distribution of $X$. We record this as the following proposition.

\begin{proposition} \label{couplingconstruction}
Let $X$ be the number of crossings in a randomly embedded graph $G_\pi$. Let $X^s$ be constructed as above. Then $X^s$ has the size-bias distribution of $X$.
\end{proposition}

\subsection{Bounding the Variance Term}

In order to apply Theorem \ref{sizebiasStein}, we need an upper bound on the variance term $\Var(E[X^s - X \mid X])$.  
We compute an upper bound in the following lemma, which is also one of the main results of this paper. 

\begin{lemma}\label{lemmaVariance}
Let $G$ be a graph on $n$ vertices and let $X$ be the number of crossings in a uniformly random embedding of $G$. Then 
\bal
\Var(E[X^s- X \mid X]) &\leq 4\Delta^2 m^2 \left( 1 + \left(\frac{32 \Delta}{m} - 1  \right)\frac{6m_4(G)}{m_2(G)^2} + \frac{\Delta^2m}{2m_2(G)} \right ),
\eal
where $m$ is the number of edges of $G$, $m_r(G)$ is the number of $r$-matchings of $G$, and $\Delta$ is the maximum degree of $G$. 
\end{lemma}

\begin{proof}
First we write the expectation as
\bal
E(X^s - X \mid X) &=  \sum_{i \in M_2(G) } E(X^s - X \mid X, I = i)P(I = i) \\
&= \frac{1}{m_2(G)} \sum_{i \in M_2(G)} \left( X^{(i)} - X \right),
\eal
where $X^{(i)}$ denotes $X^s$ conditioned to have a crossing at index $I = i$.
Thus we can write the variance as
\bal
\Var(E[X^s - X \mid X]) = \frac{1}{m_2(G)^2} \sum_{i,j \in M_2(G)} \cov(X^{(i)} - X, X^{(j)} - X). 
\eal

Let $V(i)$ denote the four vertices of the 2-matching $i$. We split the sum above into two cases according to whether or not the vertices $V(i), V(j)$ intersect. 

{\bf Case 1:} Suppose $V(i) \cap V(j) \neq \emptyset$. The number of covariance terms for this case is
\bal
|\{i,j \in M_2(G) : V(i) \cap V(j) \neq \emptyset \}| = m_2(G)^2-\binom{4}{2} m_4(G) = m_2(G)^2- 6 m_4(G).
\eal
For these terms, it is enough to use the bound
\bal
\cov(X^{(i)} - X, X^{(j)} - X) \leq E[(X^{(i)} - X)^2] \leq 4 \Delta^2 (m-1)^2,
\eal
where we used the upper bound $|X^{(i)} - X| \leq 2 \Delta(m-1)$, since by construction at most $m-1$ crossings are created for every edge which is incident to a vertex of $i$. 
Thus the contribution to the sum from these terms is upper bounded by
\bal
\frac{4 \Delta^2 (m-1)^2 \left(m_2(G)^2- 6m_4(G) \right)}{m_2(G)^2}.
\eal

{\bf Case 2:} Otherwise suppose $V(i) \cap V(j) = \emptyset$.
The number of covariance terms for this case is
\bal
|\{i,j \in M_2(G) : V(i) \cap V(j) = \emptyset \}| = \binom{4}{2} m_4(G) = 6 m_4(G).
\eal
Let $N(i)$ be the set of edges incident to the four vertices of the 2-matching $i$. We can further split this case into two subcases.

\textbf{Subcase 2.1} Suppose $V(i) \cap V(j) = \emptyset$ and $N(i) \cap N(j) \neq \emptyset$. Then the number of covariance terms for this subcase is upper bounded by
\bal
|\{i,j \in M_2(G) : V(i) \cap V(j) = \emptyset, N(i) \cap N(j) \neq \emptyset \}| \leq \frac{m_2(G)(\Delta-1)^2(m-4)}{2}.
\eal
To see this, observe that there must exist at least one edge $(v_1, v_2)$ such that $v_1 \in V(i)$ and $v_2 \in V(j)$. 
To count the number of such configurations, first choose a $2$-matching, $i$, and consider one of its four vertices $v \in V(i)$. Let $w$ be a neighbor of $v$ such that $w \in V(j)$. 
There are at most $\Delta - 1$ choices for $w$. Next, find a neighbor of $w$, which is not $v$, to define one of the edges of $j$. There are at most $\Delta - 1$ choices for this neighbor. 
Finally, we find another edge, which cannot be in $i$ nor contain $w$, to define the other edge of $j$. There are at most $m-4$ choices. 
See Figure \ref{subcasefig} for a diagram which explains this counting. Combining the above and taking into account double counting yields the desired upper bound. 

Again for these terms, it is enough to use the bound $\cov(X^{(i)} - X, X^{(j)} - X) \leq 4 \Delta^2 (m-1)^2$. Thus the contribution to the sum from these terms is upper bounded by
\bal
\frac{2\Delta^2(m-1)^2 \left(m_2(G)(\Delta-1)^2(m-4) \right)}{m_2(G)^2}.
\eal
    
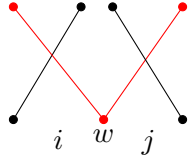
\begin{figure}
\centering 
\begin{tikzpicture}[scale = 1.5]
\draw (0,0)  -- (0.6,1) ;
\filldraw (0,0) circle (1 pt);
\filldraw (0.6,1) circle (1 pt);
\draw [red] (0,1) -- (0.8,0) ;
\filldraw [red] (0,1) circle (1 pt);
\filldraw [red] (1.5,1) circle (1 pt);
\filldraw [red] (0.8,0) node[black][below]{$w$} circle (1 pt);
\draw [red] (0.8,0) -- (1.5,1);
\draw (0.88,1)  -- (1.5,0) ;
\filldraw (0.88,1) circle (1 pt);
\filldraw (1.5,0) circle (1 pt); 
\filldraw (0.4,0) node[black][below]{$i$} circle (0 pt);
\filldraw (1.2,0) node[black][below]{$j$} circle (0 pt);
\end{tikzpicture}
\caption{A generic pair of $2$-matchings sharing a vertex.} 
\label{subcasefig}
\end{figure}

{\bf Subcase 2.2:} Suppose $V(i) \cap V(j) = \emptyset$ and $N(i) \cap N(j) = \emptyset$. For $i, k \in M_2(G)$, define the events 
\bal
A_k^{(i)} := A_k^{(i)}(G_\pi) = \{\text{$k$ is a crossing in $G_{\pi^{(i)}}$ and $k$ is not a crossing in $G_\pi$}\}, \\
B_k^{(i)} := B_k^{(i)}(G_\pi) = \{\text{$k$ is not a crossing in $G_{\pi^{(i)}}$ and $k$ is a crossing in $G_\pi$}\},
\eal
where $G_{\pi^{(i)}}$ is $G_{\pi^s}$ conditioned to have a crossing at index $I = i$, as defined in the size-bias coupling construction. Observe that 
\bal
X^{(i)} - X = \sum_{k\in M_2(G)} \I_{A_k^{(i)}} - \I_{B_k^{(i)}}.
\eal
Then for fixed $i,j \in M_2(G)$, we have that
\bal
\cov(X^{(i)} - X, X^{(j)} - X) = \sum_{k, \ell \in  M_2(G)} \cov\left( \I_{A_k^{(i)}} - \I_{B_k^{(i)}}, \I_{A_\ell^{(j)}} - \I_{B_\ell^{(j)}}\right). 
\eal

For this subcase, our strategy is to bound the number of pairs $(k,\ell)$ which have nonzero contribution to the covariance term, $\cov\left( \I_{A_k^{(i)}} - \I_{B_k^{(i)}}, \I_{A_\ell^{(j)}} - \I_{B_\ell^{(j)}}\right)$.

First observe that if $k$ and $i$ have no vertices in common, then $\I_{B_k^{(i)}}=0$ and $\I_{A_k^{(i)}}=0$ since $k$ remains equal in both $G_\pi$ and $G_{\pi^{(i)}}$. 
Thus $\I_{B_k^{(i)}}-\I_{A_k^{(i)}}=0$. Similarly if $\ell$ and $j$ have no vertices in common, then $\I_{A_\ell^{(j)}}-\I_{B_\ell^{(j)}}=0$. 
In either case, $\cov\left( \I_{A_k^{(i)}} - \I_{B_k^{(i)}}, \I_{A_\ell^{(j)}} - \I_{B_\ell^{(j)}}\right)=0$.

On the other hand, if $k$ and $j$ and have no vertices in common, and moreover $\ell$ and $i$ have no vertices in common, 
then $\cov(\I_{A_k^{(i)}} - \I_{B_k^{(i)}}, \I_{A_\ell^{(j)}} - \I_{B_\ell^{(j)}}) = 0$ 
since the random variables $\I_{A_k^{(i)}} - \I_{B_k^{(i)}}$ and $\I_{A_\ell^{(j)}} - \I_{B_\ell^{(j)}}$ have no interaction and so must be independent.

It follows that the $(k, \ell)$ which have nonzero contribution to the covariance terms satisfy the following properties:
\begin{enumerate}
\item $V(k)\cap V(i)\neq \emptyset$.
\item $V(\ell)\cap V(j)\neq \emptyset$.
\item $V(k)\cap V(j)\neq \emptyset$ or $V(\ell)\cap V(i)\neq \emptyset$.
\end{enumerate}

We claim that the number of pairs $(k,\ell)$ satisfying (1), (2) and (3) is at most $128\Delta^3m$. 
We consider the case where $V(k)\cap V(j)\neq \emptyset$. The other cases follow by exactly the same argument.

First consider the possibilities for $k$. Since $V(k)\cap V(j)\neq \emptyset$ and $V(k)\cap V(i)\neq \emptyset$, 
$k$ has one edge incident to $i$ and one edge incident to $j$ (recall that $N(i)\cap N(j)=\emptyset$). Each of these edges has at most $(4\Delta-2)$ possibilities, 
leaving at most $16\Delta^2$ possibilities for $k$.

Next consider the possibilities for $\ell$. Since $V(\ell)\cap V(j)\neq \emptyset$, $\ell$ must have one edge incident to $j$. 
This edge has at most $4\Delta-2$ possibilities, while the other one can be any other edge, having at most $m-1$. It follows that $\ell$ has at most $4\Delta m$ possibilities. 

Thus for the case $V(k) \cap V(j)\neq \emptyset$ we obtain a bound of $64\Delta^3 m$. 
Since the case $V(\ell)\cap V(i)\neq \emptyset$ is identical, we get a total bound of $128\Delta^3 m$.
It follows that $\cov(X^{(i)} - X, X^{(j)} - X) \leq 128\Delta^3 m$. 
The number of covariance terms for this subcase can be upper bounded by $6m_4(G)$, which is the number of covariance terms for all of Case 2. 
Thus the contribution to the sum from this subcase is upper bounded by $128\Delta^3 m \left( \frac{6m_4(G)}{m_2(G)^2} \right)$. 

Therefore adding up the contributions from the various cases gives an upper bound of
\bal
\Var(E[X^s- X \mid X]) &\leq \frac{4 \Delta^2 (m-1)^2 \left(m_2(G)^2- 6m_4(G) \right)}{m_2(G)^2} \\
&\qquad + \frac{2\Delta^2(m-1)^2 \left(m_2(G)(\Delta-1)^2(m-4) \right)}{m_2(G)^2} \\
&\qquad + 128\Delta^3 m \left( \frac{6m_4(G)}{m_2(G)^2} \right) \\
&\leq 4\Delta^2 m^2 \left( 1 + \left(\frac{32 \Delta}{m} - 1  \right)\frac{6m_4(G)}{m_2(G)^2} + \frac{\Delta^2m}{2m_2(G)} \right ),
\eal
as desired. 
\end{proof}

\subsection{Proof of the Main Theorem}

In this section we prove our main result. We begin with the following proposition, which provides a general Kolmogorov upper bound.

\begin{proposition} \label{generalUB}
Let $G$ be a graph on $n$ vertices and let $X$ be the number of crossings in a uniformly random embedding of $G$. Let $W = \frac{X - \mu}{\sigma}$, where $\mu = E(X)$ and $\sigma^2 = \Var(X)$. 
Then 
\bal
d_K(W, Z) \leq \frac{4 \Delta m m_2(G)}{3\sigma^2} \left( \frac{6 \Delta m}{\sigma} + \sqrt{1 + \left( \frac{32 \Delta}{m} - 1 \right)\frac{6m_4(G)}{m_2(G)^2} + \frac{\Delta^2 m}{2m_2(G)} }\right),
\eal
where $m$, $m_r(G)$, and $\Delta$ are the number of edges, the number of $r$-matchings, and the maximum degree of $G$, respectively, and $Z$ is a standard normal random variable.
\end{proposition}

\begin{proof}
Let $(X, X^s)$ be the size-bias coupling defined in Section \ref{sizebiascouplingconstruction} so that by Proposition \ref{couplingconstruction} $X^s$ has the size-bias distribution of $X$. 
Recall that $|X^s - X| \leq 2\Delta (m-1) \leq 2\Delta m$, so that $(X^s, X)$ is a bounded size-bias coupling. 
By Theorem \ref{crossingsmeanvariance}, $\mu = \frac{m_2(G)}{3}$. 

Therefore by Lemma \ref{lemmaVariance} and Theorem \ref{sizebiasStein}, 
\bal
&d_K(W, Z) \leq \frac{6m_2(G) \cdot 4\Delta^2 m^2}{3\sigma^3} \\
& + \frac{2m_2(G) \cdot 2 \Delta m}{3\sigma^2} \sqrt{1 + \left( \frac{32 \Delta}{m} - 1 \right)\frac{6m_4(G)}{m_2(G)^2} + \frac{(\Delta-1)^2(m-4)}{2m_2(G)} } \\
&\leq \frac{8\Delta^2 m^2 m_2(G)}{\sigma^3} + \frac{4\Delta m m _2(G)}{3 \sigma^2} \sqrt{1 + \left( \frac{32 \Delta}{m} - 1 \right)\frac{6m_4(G)}{m_2(G)^2} + \frac{\Delta^2 m}{2m_2(G)} } \\
&= \frac{4 \Delta m m_2(G)}{3\sigma^2} \left( \frac{6 \Delta m}{\sigma} + \sqrt{1 + \left( \frac{32 \Delta}{m} - 1 \right)\frac{6m_4(G)}{m_2(G)^2} + \frac{\Delta^2 m}{2m_2(G)} } \right). \qedhere
\eal
\end{proof}

Observe that this implies a central limit theorem, with a rate of convergence, for the number of crossings in a random embedding, provided that the upper bound converges to zero as $m$ goes to infinity. 

Our next goal is to simplify the bound in Proposition \ref{generalUB} and write it in terms of the maximum degree $\Delta$ and the number of edges $m$. We start with the following two technical lemmas, which give upper and lower bounds on several quantities which will be useful for our calculations.

\begin{lemma} \label{lemma:bounds}
Let $G$ be a graph with $m$ edges and maximum degree $\Delta$. Then
\begin{enumerate}[(i)]
\item $m_2^2(G)-6m_4(G)= 6m_3(G)+m_2(G)+4S_2 + 4S_4 + 2S_5 + 2S_6 + 2S_7,$
\item $m^2-m(2\Delta-1) \leq 2m_2\leq m^2-m$,
\item $m(m-2\Delta+1)(m-4\Delta+2) \leq 6 m_3\leq m(m-1)(m-2)$,
\item $0 \leq S_2 \leq m(\Delta-1)  (m-2)(m-3)/2$,
\item $m(m-4\Delta)/2 \leq S_4 \leq m(\Delta-1) (m-2) (\Delta-1)/2$,
\item $0 \leq S_5 \leq m (\Delta-1)^2(m-3)$,
\item $0 \leq S_6 \leq m(\Delta-1)^3$,
\item $m(m-3\Delta) \leq S_7 \leq m (\Delta-1) (m-2)$,
\end{enumerate}
where $m_r(G)$ is the number of r-matchings of $G$ and $S_i$ is the number of subgraphs of $G$ of type $C_i$, as defined in Figure \ref{secondmomentcrossingconfigs}.
\end{lemma}

\begin{proof}
\begin{enumerate}[(i)]
    \item This follows from the same argument as the proof of the variance in Theorem \ref{crossingsmeanvariance}.
    \item We can count $m_2$ directly by the following formula $$2m_2(G)=m^2+m-2\sum^n_{i=1} d_i^2,$$
where $d_i$ is the degree of the vertex $i$. Indeed, to form a matching, we choose an edge $e=(i,j)$ and consider another edge $f$ which is does not share a vertex with $e$. There are $m+1-d_i-d_j$ possibilities for $f$, and summing over all $e$ gives

$$\sum_{e=(i,j)} (m+1-d_i-d_j)=\sum_{e=(i,j)} m +\sum_{e=(i,j)}1- \sum_{e=(i,j)} d_i-\sum_{e=(i,j)} d_j=m^2+m-2\sum^n_{i=1}d_j.$$
where we note that in the sum $\sum_{e=(i,j)} d_i$, index $i$ appears $d_i^2$ times. Since we are double counting, this counts $2m_2$.

\item For the upper bound, we consider the case in which we can form a 3-matching with every possible edge. This gives $m(m-1)(m-2)$. For the lower bound, note that $m-2\Delta-1$ is the maximum possible number of edges. 
\end{enumerate}
\noindent The proof of cases $(iv)$ to $(viii)$ follows analogously. 
\end{proof}

\begin{lemma}\label{lemma:lowerboundvariance} Let $G$ be a graph with $m$ edges and maximum degree $\Delta$, and let $\sigma^2 = \Var(X)$ be the variance of the number of crossings in a uniformly random embedding of $G$. Then 
 $$\sigma^2\geq c_\Delta m^3 = \left\{ \begin{array}{cc}
     1/45 m^3, & m\geq 2 , \Delta = 1   \\
     1/45 m^3, & m \geq 61, \Delta = 2 \\
     1/90 m^3, & m\geq p(\Delta) + \sqrt{q(\Delta)}, \Delta \geq 3
 \end{array} \right.$$
where $p(\Delta) = 2\Delta^2 +8\Delta -25 $ and $q(\Delta) = 5\Delta^3-37\Delta^2 +97\Delta-25$.
\end{lemma}

\begin{proof}
From Theorem \ref{crossingsmeanvariance}, we have
\bal
\sigma^2  = \frac{2}{15}m_3(G) + \frac{2}{9}m_2(G)+ \frac{1}{45}S_4 - \frac{1}{45}S_5 - \frac{1}{18}S_6 + \frac{1}{9}S_7.
\eal
We split the proof into several cases.

{\bf Case  $\Delta =1$:} If the maximum degree is one, then the graph is necessarily a matching graph. Let $G$ be the union of $m$ copies of $K_2$. Note that we only need to consider $m_3(G)$ and $m_2(G)$ since the other types of subgraphs are not present in $G$. Then 
\bal
\sigma^2 = \frac{2}{15}m_3(G) + \frac{2}{9}m_2(G) = \frac{2}{15}\binom{m}{3} + \frac{2}{9}\binom{m}{2} =\frac{m(m-1)(m+3)}{45},
\eal
which is larger that $m^3/45$ for all $m \geq 2$.

Next, we consider $\Delta>1$. For this, note that from Lemma \ref{lemma:bounds}, the variance can be lower bounded by
\bal
\sigma^2 &\geq \frac{m^3}{45}- \frac{2\Delta^2+8\Delta-25}{90}m^2 - \frac{5\Delta^3-37\Delta^2 +97\Delta-25}{90}m \\
&= \frac{m^3}{45}- \frac{p(\Delta)}{90}m^2 - \frac{q(\Delta)}{90}m. 
\eal

{\bf Case  $\Delta = 2$:} In this case, the previous upper bound is
\bal
\sigma^2 &\geq \frac{m^3}{45} + \frac{1}{90}m^2 - \frac{61}{90}m,
\eal
which is larger that $m^3/45$ for $m \geq 61$.

{\bf Case  $\Delta \geq 3$:} Observe that in this case, both $p(\Delta)$ and $q(\Delta)$ are positive. Then from 
\bal
\frac{m^3}{45}- \frac{p(\Delta)}{90}m^2 - \frac{q(\Delta)}{90}m  \geq \frac{m^3}{90},
\eal 
it follows that
\bal 
m^2 - p(\Delta)m - q(\Delta) \geq 0.
\eal
Moreover the largest root, $r$, of the polynomial $m^2 - p(\Delta)m - q(\Delta)$ satisfies 
\bal
r = \frac{p(\Delta) + \sqrt{p(\Delta)^2+4q(\Delta)}}{2} \leq p(\Delta) + \sqrt{q(\Delta}). 
\eal
Thus we can conclude that if $m \geq p(\Delta) + \sqrt{q(\Delta}$, then $\sigma^2 \geq m^3 /90$.
\end{proof}

We are now in a position to prove our main theorem. 

\begin{proof}[Proof of Theorem \ref{maintheorem}]
It suffices to find an upper bound for the general Kolmogorov upper bound in Proposition \ref{generalUB} in terms of $m$ and $\Delta$. To this end, first observe that
\bal 
& 1 + \left( \frac{32 \Delta}{m} - 1 \right)\frac{6m_4(G)}{m_2(G)^2} + \frac{\Delta^2 m}{2m_2(G)} \\
& \qquad = \frac{2m(m_2(G)^2-6m_4(G)) + 384\Delta m_4(G)+\Delta^2 m^2 m_2(G)}{2m m_2(G)^2}.
\eal
Applying Lemma \ref{lemma:bounds}, we have that
\bal
384\Delta m_4(G)+\Delta^2 m^2  \cdot m_2(G) &\leq 16\Delta( m^4 -6m^3+11m^2-6m) + \Delta^2 m^2 \frac{(m^2-m)}{2}\\
    &= \left(16\Delta + \frac{\Delta^2}{2} \right)m^4 -\left(96\Delta +\frac{\Delta^2}{2}\right)m^3 +176 \Delta m^2 -96 \Delta m,
\eal
and 
\bal
&2m(m_2(G)^2-6m_4(G)) \\ &\qquad \leq \left(4 \Delta -2\right) m^4 + \left(8\Delta^2-32\Delta +19\right)m^3 + \left( 4\Delta^3-32\Delta^2 +68\Delta - 37 \right) m^22,
\eal
from which it follows that
\begin{align*}
     &1 + \left( \frac{32 \Delta}{m} - 1 \right)\frac{6m_4(G)}{m_2(G)^2} + \frac{\Delta^2 m}{2m_2(G)} \\
     &\leq \frac{( \Delta^2/2+ 20\Delta -2)m^4 +(15\Delta^2/2 - 128 \Delta+19 )m^3 +(4\Delta^3-32\Delta^2 +244\Delta - 37)m^2 -96 \Delta m}{2m m_2(G)^2} \\
     &= \frac{( \Delta^2/2+ 20\Delta -2)m^3 +(15\Delta^2/2 - 128 \Delta+19 )m^2 +(4\Delta^3-32\Delta^2 +244\Delta - 37)m -96 \Delta }{2m_2(G)^2} \\
     &\leq \frac{( \Delta^2+ 40\Delta -4)m^3 +(15\Delta^2 - 256 \Delta+ 38 )m^2 +2(4\Delta^3-32\Delta^2 +244\Delta - 37)m - 192 \Delta }{m^2 (m-2\Delta+1)^2} \\
     &\leq \frac{2( \Delta^2+ 40\Delta -4)m^3}{m^2 (m-2\Delta+1)^2}.
\end{align*}
Therefore, for $m \geq (3+ \sqrt{6})(2\Delta -1)$, 
 \bal 
 1 + \left( \frac{32 \Delta}{m} - 1 \right)\frac{6m_4(G)}{m_2(G)^2} + \frac{\Delta^2 m}{2m_2(G)} &\leq  \frac{2( \Delta^2+ 40\Delta -4)m}{(m-2\Delta+1)^2} \\ &\leq \frac{3(\Delta^2+ 40\Delta -4 )}{m} \\ &\leq \frac{3\Delta(\Delta+40)}{m}.
\eal
Moreover from Lemmas \ref{lemma:bounds} and \ref{lemma:lowerboundvariance}, 
\bal
\frac{4\Delta m m_2(G)}{3\sigma^2} \leq \frac{2 \Delta}{3 c} \qquad \text{and} \qquad \frac{6\Delta m}{\sigma} \leq \frac{6\Delta}{\sqrt{c_{\Delta}m}}.
\eal



Therefore plugging these bounds into the general Kolmogorov upper bound given in Proposition \ref{generalUB} and applying the AM-GM inequality yields
\bal
d_K(W, Z) &\leq \frac{2 \Delta}{3c_{\Delta}} \left( \frac{6\Delta}{\sqrt{c_{\Delta} m}} + \sqrt{ \frac{3\Delta(\Delta+40)}{m} }\right) \\ 
&\leq  \frac{1}{\sqrt{m}} \left(\frac{(12+4\sqrt{c_{\Delta}})\Delta^2 + 40\sqrt{c_{\Delta}}\Delta }{3c_{\Delta}^{3/2}} \right). 
\eal
Now, for $\Delta \in \{1, 2\}$,
\begin{equation}\label{bounddelta1}
d_K(W, Z) \leq  \frac{1}{\sqrt{m}} \left(60(1+9\sqrt{5})\Delta^2 + 600 \Delta \right)
\end{equation}
for all $m \geq 61$, while for $\Delta \geq 3$, 

\begin{equation}\label{bounddelta2}
d_K(W, Z) \leq  \frac{1}{\sqrt{m}} \left(120(1+9\sqrt{10})\Delta^2 + 1200 \Delta \right)
\end{equation}
for all $m \geq p(\Delta) + \sqrt{q(\Delta)}$. 

 Finally, note that for $\Delta\in\{1,2\}$ if $m\leq 61$, then the right-hand side of \eqref{bounddelta1} is  larger than 1, so the inequality holds for all $m$.  Similarly, for $\Delta\geq3$ if $m\leq 4\Delta ^2$, the right-hand side of \eqref{bounddelta2} is larger than 1. Since $4\Delta^2\geq p(\Delta) + \sqrt{q(\Delta)}$, one concludes the desired inequality for all $m$.
\end{proof}

\section{Deterministic Examples} \label{Examples}

In this section we apply the general Kolmogorov bound, Proposition \ref{generalUB}, to prove central limit theorems, with convergence rates, for the number of crossings in matchings, path graphs, cycle graphs, and the disjoint union of $n$ triangles. Throughout this section, let $X$ be the number of crossings in a random embedding of $G$, the graph under consideration. 

\subsection{Matchings} \label{matchings_example}
Let $G$ be a disjoint union of $n$ copies of $K_2$ graphs. This is called a {\em matching} on $2n$ vertices or a {\em chord diagram} of size $n$. Chord diagrams have found applications in a wide range of fields such as topology, biology, physics, and free probability \cite{APRW13, BNS18, BR00, CP13, HZ86, Kon93, LN11, MY13, NS06}. 

The central limit theorem for the number of crossings in a random chord diagram, without a convergence rate, was first proven by Flajolet and Noy \cite{FN00} using generating functions and analytic combinatorics. 
We apply Theorem \ref{maintheorem} to obtain the first convergence rate result for this central limit theorem.

The expected number of crossings is given by
\bal
EX=\frac{1}{3} m_2(G)=\frac{n(n-1)}{6}.
\eal
For the variance, we only need to consider $m_2(G)$ and $m_3(G)$, since the other types of subgraphs are not present in $G$. 
The number of $r$-matchings is given by $m_r(G)=\binom{n}{r}$, since any choice of $r$ edges constitutes an $r$-matching. Thus by Theorem \ref{crossingsmeanvariance} we get
\bal
\mathrm{Var}(X)=\frac{2}{15} \binom{n}{3} + \frac{2}{9} \binom{n}{2} = \frac{n(n-1)(n+3)}{45},
\eal
which is larger that $n^3/45$ for all $n \geq 2$. 

Since $\Delta =1$ and $m = n$, Theorem \ref{maintheorem} gives a Kolmogorov bound of
\bal
d_K(W,Z) \leq \frac{1868}{\sqrt{n}}
\eal
for all $n \geq 2$. In particular, by sending $n$ to infinity we recover the Flajolet-Noy  central limit theorem \cite{FN00}, along with an explicit rate of convergence.


\begin{rem}
In the case of matchings, one can actually calculate precisely the contribution of Case 2 in the proof of Lemma \ref{lemmaVariance}. This is done by a detailed  analysis of all the possibilities of the four $2$-matchings involved in the proof. This yields the bound
$$\Var(E[X^s- X \mid X]) \leq \frac{3020n}{189} + \frac{37792}{525(n-1)} + \frac{37672}{1575n} + \frac{151664}{4725} \approx16n.$$ 
Using this quantity instead, one can slightly improve our bound to $1268n^{-1/2}$ for $n$ sufficiently large. 
\end{rem}

\subsection{Paths} \label{path_example}
Let $P_n$ be the path graph on $n$ vertices. The number of $r$-matchings is $m_r(P_n) = \binom{n-r}{r}$, so that 
\bal
&m_2=\binom{n-2}{2}, \quad   m_3=\binom{n-3}{3},  \quad S_4=  \binom{n-4}{2},  \\
&S_5 = 2\binom{n-4}{2},  \quad S_6 = n-4,  \quad S_7 =  2 \binom{n-3}{2}.
\eal
Then
\bal
\Var(X) &= \frac{2}{15} \binom{n-3}{3} + \frac{2}{9} \binom{n-2}{2} - \frac{1}{45}\binom{n-4}{2}-\frac{n-4}{18} +\frac{2}{9}\binom{n-3}{2} \\
&= \frac{n^3}{45}-\frac{n^2}{18}-\frac{11n}{45}+\frac{2}{3},
\eal
which is larger that $n^3/60$ for $n \geq 14$. 

Since $\Delta = 2$ and $m = n-1$, Theorem \ref{maintheorem} yields 
\bal
d_K(W,Z) \leq \frac{6270}{\sqrt{n-1}}
\eal
for all $n \geq 2$.

\subsection{Cycles} \label{cycle_example}
Let $C_n$ be the cycle graph on $n$ vertices. The number of $r$-matchings is $m_r(C_n)=\frac{n}{r}\binom{n-r-1}{r-1}$, then
$$EX=\frac{m_2}{3}=\frac{n(n-3)}{6}.$$
The subgraph counts are
\bal
&m_2=\frac{n(n-3)}{2}, \quad m_3=\frac{n}{3}\binom{n-4}{2}, \quad S_4 =\frac{n(n-5)}{2}, \\
&S_5 = n(n-5), \quad S_6 = n, \quad S_7 = n(n-4).
\eal
The variance is 
\bal
\Var(X) &= \frac{2n}{45} \binom{n-4}{2} + \frac{2n(n-3)}{18}  - \frac{n(n-5)}{90}-\frac{n}{18} +\frac{n(n-4)}{9}\\
&= \frac{n^3}{45}+\frac{ n^2}{90} - \frac{n}{3},
\eal
which is larger that $n^3/50$ for $n \geq 15$. 

Since $\Delta = 2$ and $m = n$, Theorem \ref{maintheorem} gives
\bal
d_K(W,Z) \leq \frac{6270}{\sqrt{m}} = \frac{6270}{\sqrt{n}}
\eal
for all $n \geq 3$.


\subsection{Triangles} \label{triangle_example}
Let $G$ be the disjoint union of $n$ copies of $K_3$. 
In this case, the subgraphs of type $S_5$ and $S_6$ are not present in $G$. To obtain an $r$-matching for $r \geq 2$, we first choose $r$ triangles, 
and for each triangle we choose one of the three edges to include in the $r$-matching. Thus the number of $r$-matchings is $m_r(G)=3^r \binom{n}{r}$. This gives
\bal
m_2 = 3^2 \binom{n}{2}, \quad m_3 = 3^3 \binom{n}{3}, \quad S_4  = 3^2 \binom{n}{2}, \quad S_7 = 2 \cdot 3^2 \binom{n}{2}.
\eal
The variance is 
\bal
\Var(X) &= \frac{18}{5} \binom{n}{3} + 2 \binom{n}{2}+ \frac{1}{5}\binom{n}{2} + 2 \binom{n}{2} = \frac{3n^3}{5}+\frac{3n^2}{10}-\frac{9n}{10},
\eal
which is larger than $3n^3/5$ for $n \geq 3$. 

Since $\Delta = 2$ and $m = 3n$, Theorem \ref{maintheorem} gives 
\bal
d_K(W,Z) \leq \frac{6270}{\sqrt{m}} \leq \frac{3620}{\sqrt{n}}
\eal
for all $n \geq 1$.


\section{Random Graph Examples} \label{RandomExamples}

Previously we have only considered random embeddings of deterministic graphs. In this section, we consider crossings in a uniformly random embedding of a {\em random graph}. Suppose $\{G_{n,i}\}_{i \in I}$ is a sequence of deterministic graphs on $n$ vertices, for some countable index set $I$. For each $n$, let $\Gc_n$ be the random graph which equals $G_{n,i}$ with probability $p_i$, so that $p_i = P(\Gc_n = G_{n,i})$.

Let $X := X(\Gc_n)$ be the number of crossings in a uniformly random embedding of the random graph $\Gc_n$. Similarly let $X_i := X(G_{n,i})$ be the number of crossings in a uniformly random embedding of the deterministic graph $G_{n,i}$. Conditioning on the event $\{\Gc_n = G_{n,i}\}$, we get
\bal
P(X = k)=\sum_{i \in I}  p_i P( X_i = k).
\eal
In particular, $X$ is a mixture of the distributions of $X_i$, so that $F_X = \sum_{i \in I} p_i F_{X_i}$.

\subsection{Kolmogorov Bounds for Mixture Distributions}

Given a countable set of cumulative distribution functions $\{F_i\}^\infty_{i=1}$, and a set of nonnegative weights  $\{p_i\}^\infty_{i=1}$,  such that $\sum_i p_i=1$, the distribution function
$F(x)=\sum_{i=1}^n p_i F_i(x)$
is called a mixture or mixture distribution. If $F_i$ is Gaussian distributed for all $i$, then $F$ is a Gaussian mixture distribution.

We start with some useful lemmas on mixture distributions. The first is a standard result which gives an upper bound on the Kolmogorov distance between two mixture distributions.

\begin{lemma} \label{mixturedistance}
Let $\{\mu_k\}_{k=1}^n$ and $\{\nu_k\}_{k=1}^n$ be collections of measures with distribution functions $\{F_k\}_{k=1}^n$ and $\{G_k\}_{k=1}^n$, respectively. Let $\{p_k\} \subseteq \R$ be a collection of weights such that $p_k \geq 0$ and $\sum_{k=1}^n p_k = 1$. Then
\bal
d_K\left(\sum_{k=1}^n p_k F_k, \sum_{k=1}^n p_k G_k \right) \leq \sum_{k=1}^n p_k d_K(F_k, G_k)\leq \max_{1 \leq k \leq n} d_K(F_k,G_k).
\eal
\end{lemma}

Combining this lemma with with Theorem \ref{maintheorem} allows us to prove Theorem \ref{maintheoremrandom}. 

\begin{proof}[Proof Theorem \ref{maintheoremrandom}]
Let $\Gc$ be a random graph such that $p_i = P(\Gc = G_i)$, for all $i \in I$, and let $X = X(\Gc)$ be the number of crossings in a uniformly random embedding of $\Gc$. 

Let $X_i$ denote the distribution of $X$ conditional on the event $\{\Gc = G_i\}$ and let $Z_i \sim N(\mu_i, \sigma_i^2)$ for all $i \in I$. Then for all $i \in I$, Theorem \ref{maintheorem} gives
\bal
d_K(F_{X_i},F_{Z_i}) \leq \frac{C}{\sqrt{m}}
\eal
for some constant $C = C(\Delta)$, possibly depending on $\Delta$.
Since $F_X=\sum_{i \in I} p_iF_{X_i}$, by Lemma \ref{mixturedistance} we get that
\bal 
d_K\left(F_X, \sum_{i \in I} p_i F_{Z_i} \right) &= d_K\left(\sum_{i \in I} p_iF_{X_i}, \sum_{i \in I} p_i F_{Z_i} \right) \leq \max_{i \in I} d_K(F_{X_i},F_{Z_i}) \leq \frac{C}{\sqrt{m}},
\eal 
as desired.
\end{proof}

Next, we consider a Kolmogorov bound for mixtures of Gaussian distributions. 

\begin{lemma}[Theorem 2.1, \cite{BE21}] \label{conditionalnormalapprox}
Let $X$ be a real-valued random variable such that $E(X^2) < \infty$, and let $\Gc$ be a $\sigma$-algebra such that the conditional distribution of $X$ given $\Gc$ is normal. Then
\bal
&d_K\left( \frac{X - E(X)}{\sqrt{E(\Var(X \mid \Gc)}}, N(0,1) \right) \\
&\leq \frac{ \sqrt{\Var(\Var(X \mid \Gc))} + \Var(E[X \mid \Gc]) + \sqrt{\frac{2}{\pi}} \sqrt{\Var(\Var(X \mid \Gc))} \Var(\Var(X \mid \Gc))^{1/4} }{\sqrt{E[\Var(X \mid \Gc)]}}.
\eal
\end{lemma}

As a consequence, we have the following lemma. 

\begin{lemma}[Remark 2.1, \cite{BE21}] \label{normalapproxmixtures}
For two normal random variables $N(\nu, \tau^2)$ and $ N(\mu, \sigma^2)$ the following inequality holds
\bal
d_K(N(\nu, \tau^2), N(\mu, \sigma^2)) \leq \frac{1}{\sigma^2}|\sigma^2 - \tau^2| + \frac{\sqrt{2\pi}}{4\sigma}|\mu - \nu|. 
\eal
\end{lemma}
Combining Lemmas \ref{mixturedistance} and \ref{normalapproxmixtures} yields the following, which gives a bound for the approximation of Gaussian mixtures to Gaussian distributions. 

\begin{lemma} \label{normalmixturetonormal}
Let $Z \sim N(\mu, \sigma^2)$ and let $\{Z_k\}_{k=1}^n$ be a sequence of normal random variables such that $Z_k \sim N(\mu_k, \sigma_k^2)$. Let $\{p_k\} \subseteq \R$ be a collection of weights such that $p_k \geq 0$ and $\sum_{k=1}^n p_k = 1$. Then 
\bal
d_K\left(\sum_{k=1}^n p_kZ_k, Z\right) \leq \max_k \left( \frac{|\sigma^2 - \sigma_k^2|}{\sigma^2} + \frac{\sqrt{2\pi} |\mu - \mu_k|}{4\sigma} \right).
\eal
\end{lemma}

\subsection{Finite Mixture of Matchings and Triangles}

Let $\Gc$ be a random graph which equals $M_n$, a matching on $2n$ vertices, with probability $p$, or $T_n$, a disjoint union of $n$ copies of $K_3$, with probability $1-p$. Let $X := X(\Gc)$ be the number of crossings in a random embedding of $\Gc$. Define $X_M = X(M_n)$ and $X_T = X(T_n)$ analogously. Let $\mu_M, \sigma_M^2$ (resp. $\mu_T, \sigma_T^2$) be the mean and variance of $X_M$ (resp. $X_T$). Let $Z_M \sim N(\mu_M, \sigma_M^2)$ and $Z_T \sim N(\mu_T, \sigma_T^2)$. Note that we can write
\bal
F_X = pF_{X_M} + (1-p)F_{X_T},
\eal

Using Lemmas \ref{mixturedistance} and the Kolmogorov bounds for matchings and triangles from the previous section, we get
\bal
d_K\left( F_X, pF_{Z_M}+(1-p)F_{Z_T} \right) \leq \max\left\{ d_K(F_{X_M}, F_{Z_M})), d_K(F_{X_T}, F_{Z_T}) \right\} \leq \frac{3620}{\sqrt{n}}.
\eal

Therefore combining the above and applying the triangle inequality again yields a final Kolmogorov bound of
\bal
d_K\left( \frac{X - \mu(X)}{\sigma(X)}, N(0,1) \right) \leq \frac{3620}{\sqrt{n}} + O\left( \frac{1}{n} \right),
\eal
which implies the asymptotic normality of $X$ as $n$ goes to infinity. 

\begin{figure}[ht]
\centering
\includegraphics[scale = 0.7]{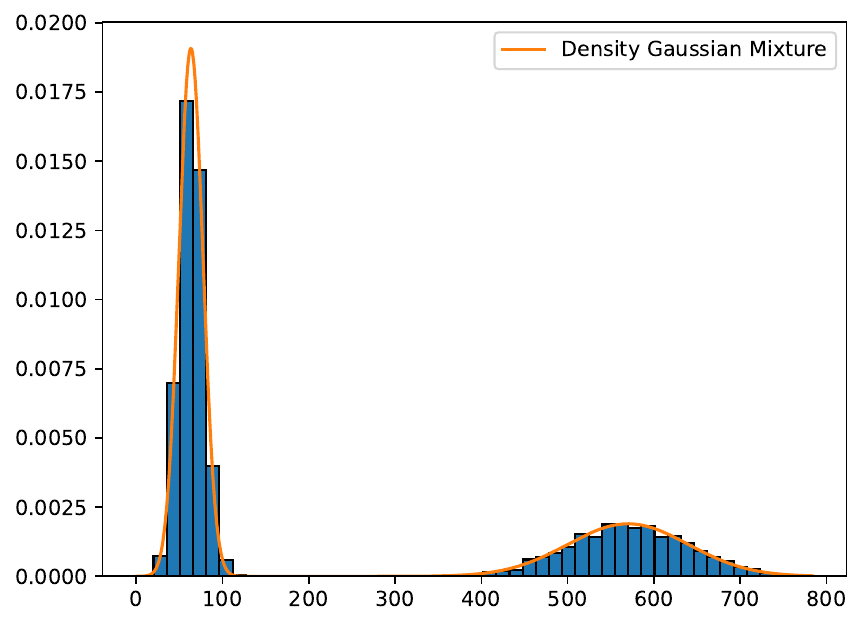}
\caption{Histogram of the number of crossings for a mixture of matchings on $n$ edges and $n$ copies of triangles, with $n=20$ and $p=1/3$ probability of generating a triangle.}
\label{mixture_example}
\end{figure}

\subsection{Random \texorpdfstring{$d$}{d}-Regular Graphs} \label{dRegular}

We consider the number of crossings in a random embedding of a uniformly random $d$-regular graph on $n$ vertices and show that it converges in distribution to a normal random variable as $n$ goes to infinity. 

Here, $\Delta = d$ and $m = \frac{nd}{2}$ (so necessarily $nd$ is even). The idea is that the number of crossings of every fixed $d$-regular graph satisfies a central limit theorem by Theorem \ref{maintheorem}. Then we can consider a mixture and bound the distance depending on the mean and variance of each fixed realization. 

Let $N_n$ be the number of $d$-regular graphs on $n$ vertices. Let $\{G_k\}_{k=1}^{N_n}$ denote the collection of all such graphs, listed in some fixed but arbitrary order. First we calculate the mean and estimate the variance of the number of crossings in a fixed $d$-regular graph, say $G_k$. It is clear that $m_2(G_k)=m(m-2d+1)/2$ and so $$\mu := E(X(G_k))=\frac{m(m-2d-1)}{6}.$$ 
for all $d$-regular graphs. 
For the variance,  we first give a small improvement of Lemma \ref{lemma:bounds} in the case of $d$-regular graphs.

\begin{lemma} \label{lemma:bounds2}
Let $G$ be a regular graph with $m$ edges and degree $d$. Then
\begin{enumerate}[(i)]
\item $2m_2= m(m-2d+1)$,
\item $m_2(m-4d+2) \leq 3 m_3\leq m_2(m-4d+6)$,
\item $m_2(d-2)^2 \leq S_4 \leq m_2(d-1)^2$,
\item $m_2(d-3)(d-4)\leq S_5 \leq m_2(d-1)^2$,
\item $m(d-1)(d-2)(d-3)\leq S_6 \leq m(d-1)^3$,
\item $m_2(2d-5) \leq S_7 \leq m_2 (2d-2) $,
\end{enumerate}
\end{lemma} 
Recalling the formula
$$
\sigma^2(X_k) := \Var(X(G_k)) = \frac{2}{15}m_3(G) + \frac{2}{9}m_2(G)+ \frac{1}{45}S_4 - \frac{1}{45}S_5 - \frac{1}{18}S_6 + \frac{1}{9}S_7,$$
we obtain the above 
$$\frac{-3 + d + 15 d^2 - 5 d^3}{90} m+ \frac{-3 - 2 d}{45} m^2 \leq
\sigma^2(X_k)\leq
\frac {32 - 45 d + 2 d^2 - 5 d^3}{180} m + \frac{1 + d}{30} m^2$$
From where one can obtain that $|\sigma^2(X_k)-m^3/45|\leq \frac{6d}{45}m^2$ for $m\geq d^2$. Note that $m\geq d^2$ is equivalent to $n\geq2d$, which is quite weak since $n\geq d-1$, for any $d$-regular graph . 

Let $Z \sim N\left(\mu, \frac{m^3}{45}\right)$ and $Z_k \sim N\left( \mu, \sigma^2(X_k) \right)$. By Lemma \ref{normalapproxmixtures}, 
\bal
d_K\left(N\left(\mu(X), \sigma^2(X)\right), Z \right) \leq \frac{|\sigma^2(X)-m^3/45|}{m^3/45} \leq \frac{4dm^2/45}{m^3/45} = \frac{4d}{m},
\eal
and moreover by Theorem \ref{maintheorem},
\bal
d_K\left( X_k, N\left(\mu, \sigma^2(X_k)\right) \right) \leq \frac{120}{\sqrt{m}} \left((1+9\sqrt{10})d^2 + 10 d \right).
\eal



Finally, consider the mixture $X = \sum_{k=1}^{N_n} \frac{1}{N_n} X_k$, which is the number of crossings in a random embedding of a uniformly random $d$-regular graph on $n$ vertices. Moreover consider a Gaussian mixture $\sum_{k=1}^{N_n} \frac{1}{N_n} Z_k$, which will approximate $X$. By Lemma \ref{normalmixturetonormal}, 
\bal
d_K \left(\sum_{k=1}^{N_n} \frac{1}{N_n} Z_k, Z\right) \leq \sum_{k=1}^{N_n} \frac{1}{N_n}  \frac{|\sigma^2-\sigma^2(X_k)|}{\sigma^2} \leq \sum_{k=1}^{N_n} \frac{1}{N_n} \frac{4d}{m} \leq \frac{4d}{m}, 
\eal
and on the other hand, 
\bal
d_K\left( X, \sum_{k=1}^{N_n} \frac{1}{N_n} Z_k\right) \leq  \frac{120}{\sqrt{m}} \left((1+9\sqrt{10})d^2 + 10 d \right).
\eal
for $m>d^2$. Putting everything together and applying the triangle inequality, we obtain the following Kolmogorov bound
\bal
d_K(X, Z) \leq \frac{120}{\sqrt{m}} \left((1+9\sqrt{10})d^2 + 10 d \right) + \frac{4d}{m}\leq \frac{C}{\sqrt{n}},
\eal
for some constant $C$ and $n>2d$. Sending $n$ to infinity establishes the asymptotic normality of $X$.

\subsection{Continuous Mixture of Matchings and Triangles}

Let $G_{n,k}$ be the disjoint union of $k$ matchings and $n-k$ triangles.
For fixed $k$, the number of edges of $G_{n,k}$ is given by $m=k+3(n-k)=3n-2k$.
The number of matchings is given by 
$$m_2= \binom{k}{2}+3 k(n-k)+9\binom{n-k}{2}=\frac{8 k + 4 k^2 - 9 n - 12 k n + 9 n^2}{2}$$
which gives us the expected value
\begin{eqnarray*} E(X)&=& \frac{8 k + 4 k^2 - 9 n - 12 k n + 9 n^2}{6}.\end{eqnarray*}
To calculate the variance, we note that $S_5, S_6=0$ and
\begin{eqnarray*}m_3&=&\binom{k}{3}+3\binom{k}{2}(n-k)+  9k\binom{n-k}{2}+27\binom{n-k}{3}
\\
S_4&=&9\binom{n-k}{2}, \qquad
S_7=3k (3(k-1)+(n-k)), \\
\end{eqnarray*}
so that
\begin{eqnarray*} 
\Var(X)&=&\frac{-105 k + 13 k^2 - 16 k^3 + 9 n + 144 k n + 72 k^2 n - 63 n^2 - 
 108 k n^2 + 54 n^3}{90}.
\end{eqnarray*}

Consider $k\approx \lfloor tn \rfloor$ for $t\in[0,1]$, so that $t=0$ corresponds to the case of all matchings and $t=1$ the case of all triangles.
Then $$E(X)= \frac{(3 - 2 t)^2}{6} n^2 + \frac{-9 + 8 t}{6} n= \frac{(3 - 2 t)^2}{6}+o(n^2).$$
and
$$\Var(X)= \frac{(3 -2 t)^3}{45} n^3 + \frac{-63 + 144 t + 13 t^2}{90} n^2+ \frac{9 - 105 t}{90} n =\frac{(3 -2 t)^3}{45}n^3+ o(n^3) .$$

Consider a random graph $\mathcal{G}_n$ with $P(\mathcal{G}_n=G_{n,tn})=\frac{1}{n+1}$ and, as before, let $X$ denote the number of crossings of a random embedding of $\mathcal{G}_n$. Theorem 1 says that $X|\mathcal{G}_n =G_{n,tn}$ approaches a normal random variable with parameters $\sigma^2 \approx\sigma_t^2= n^3\frac{(3 -2 t)^3}{45}$ and $\mu\approx \mu_t=n^2 \frac{(3 - 2 t)^2}{6}$. Thus $X$ approximates a Gaussian mixture with density
\begin{equation} \label{intgaussian} f(x)=\int_0^1 \frac{1}{\sqrt{2\pi\sigma_t^2}} e^{-\frac{(x - \mu_t)^2}{2\sigma_t^2}} dt.
\end{equation}

Consider the normalized random variable $Y_n=n^{-2}X$. Then $Y_n$ converges to a random variable $Y$ with density
\begin{equation} \label{limitintgaussian}f_{Y}(x)= \sqrt{\frac{3}{8x}}, \qquad    x\in[1/6,3/2].
\end{equation}

\begin{figure}[ht]
\centering
\includegraphics[scale = 0.7]{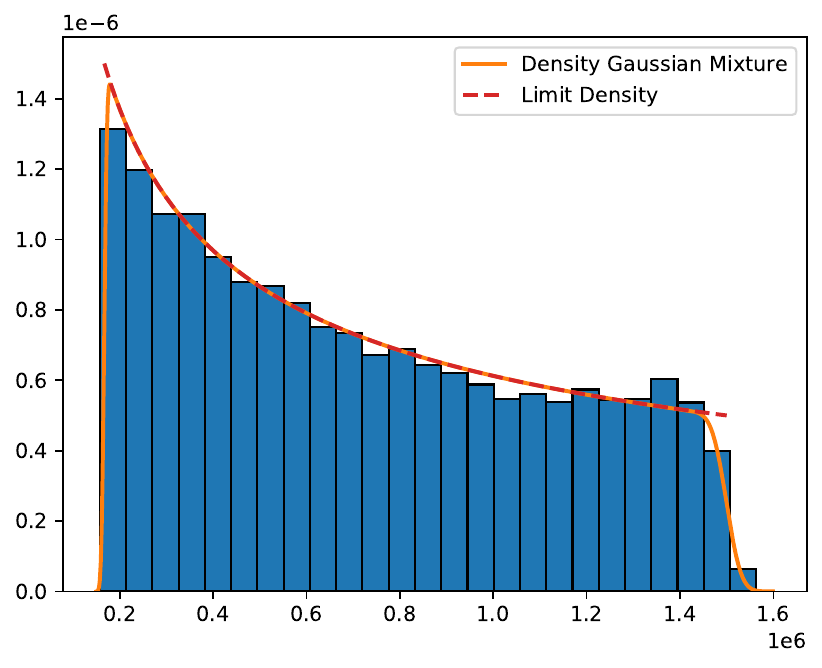}
\caption{Histogram of the number of crossings for the random graph $\mathcal{G}_{1000}$. The red dashed line corresponds to the limiting density $f_{Y}$, \eqref{limitintgaussian}, and the solid orange line corresponds to the integral Gaussian mixture, \eqref{intgaussian}.}
\label{continous_mixture_example}
\end{figure}

\section{Another Possible Limit} \label{AnotherLimit}

We consider an example which shows that not every sequence of graphs satisfies a central limit theorem for the number of crossings, 
even if the variance is not always $0$ and that having $m_2$ going to infinity is not enough. Moreover, it shows that we can have another type of limit for the number of crossings.

Consider the graph $G_n$ which consists of a star graph with $n-1$ vertices, with an edge attached to one of its leaves; see Figure \ref{kite}(a). 
Let $X_n$ be the number of crossings in a random embedding of $G_n$.
Note that in this case, $m = n-1$, $m_2 = n-3$, and $m_4 = 0$, and the only other term appearing in the calculation of $E (X_n^2)$ is  $S_7$, which for this graph equals $\binom{n-3}{2}$. 
Also note that $\Delta = n-2$. This gives 
\bal
E(X_n)=\frac{n-3}{3} \qquad \text{and} \qquad E(X_n^2)=\frac{(n-2)(n-3)}{6},
\eal
from which it follows that $\sigma^2=n(n-3)/18$. 

By Theorem \ref{maintheorem}, the upper bound on the Kolmogorov distance between the standardized number of crossings, $W_n$, and a standard normal, $Z$, is given by 
\bal
d_K(W_n, Z) \leq \frac{1}{\sqrt{n-1}} \left(120(1+9\sqrt{10})(n-1)^2 + 1200 (n-1) \right) = O(n^{3/2}),
\eal
which diverges to infinity as $n \to \infty$. 

On the other hand, we can explicitly calculate the probability of having exactly $k$ crossings. Let $v_0$ denote the central vertex of the star graph, let $v_n$ be the vertex at distance $2$ from $v_0$, 
and let $v_{n-1}$ be the vertex between $v_0$ and $v_n$. The number of crossings in a random embedding of $G_n$ depends only on the positions of these three vertices. 
More precisely, there will be exactly $k$ crossings if the following two conditions are satisfied (see Figure \ref{kite}(b) for an example):
\begin{enumerate}
\item There are exactly $k$ and $n-2-k$ vertices on the two arcs which are defined when removing $v_n$ and $v_{n-1}$.
\item Vertex $v_0$ is on the arc containing the $n-2-k$ vertices.
\end{enumerate}

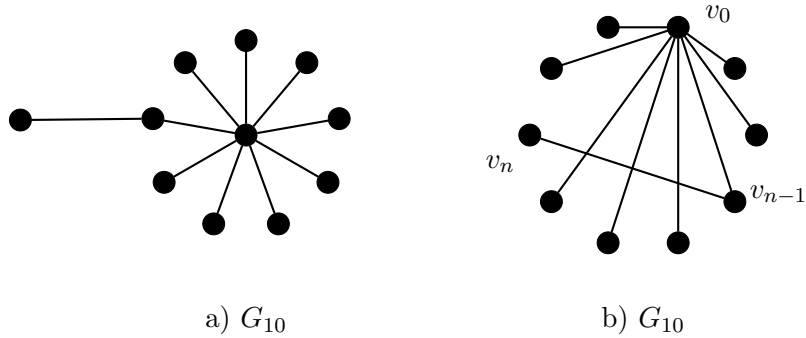
\begin{figure}[ht]
\centering 
\begin{tikzpicture}
[mystyle/.style={scale=0.8, draw,shape=circle,fill=black}]
\def\ngon{9}
\node[regular polygon,regular polygon sides=\ngon,minimum size=2.5cm] (p) {};
\foreach\x in {1,...,\ngon}{\node[mystyle] (p\x) at (p.corner \x){};}
\node[mystyle] (p0) at (0,0) {};
\foreach\x in {1,...,\ngon}
{
 \draw[thick] (p0) -- (p\x);
}
\node[mystyle] (p00) at (-3,.2) {};
{
 \draw[thick] (p00) -- (p3);
}

\node [label=below:a) $G_{10}$] (*) at (0,-2) {};
 \end{tikzpicture}
\qquad \qquad
\begin{tikzpicture}
[mystyle/.style={scale=0.8, draw,shape=circle,fill=black}]
\def\ngon{10}
\node[regular polygon,regular polygon sides=\ngon,minimum size=3cm] (p) {};
\foreach\x in {1,...,\ngon}{\node[mystyle] (p\x) at (p.corner \x){};}

\foreach\x in {5,...,\ngon}
{
 \draw[thick] (p1) -- (p\x);
}

\foreach\x in {1,...,3}{
 \draw[thick] (p1.center) -- (p\x.center);
}

\draw[thick] (p4) -- (p8);

\node [label=below: $v_{0}$] (*) at (1,2) {};
\node [label=below: $v_{n-1}$] (*) at (1.8,-.4) {};
\node [label=below: $v_{n}$] (*) at (-1.9,0) {};

\node [label=below:b) $G_{10} $ in convex position] (*) at (0,-2) {};
 \end{tikzpicture}
 
\caption{(a) The graph $G_{10}$ drawn without crossings. (b) The graph $G_{10}$ drawn in convex position with $3$ crossings formed by the edge $(v_{n-1},v_{n})$ and edges $(v_0,v_i)$, 
where $v_i$ are between $v_{n-1}$ and $v_{n}$ in the cyclic order.} 
\end{figure} \label{kite}

By a simple counting argument, since all permutations occur with equal probability, we have that the two conditions above are satisfied with probability
\bal
P(X_n = k) = \frac{2(n-2-k)}{(n-1)(n-2)}, \qquad \text{for $k=0,\ldots,n-2$}.
\eal

Define $Y_n := \frac{X_n}{n}$. Then
\bal
P\left(Y_n = \frac{k}{n}\right) = \frac{2(n-2-k)}{(n-1)(n-2)} \approx \frac{2}{n}\left(1-\frac{k}{n}\right), \qquad \text{for $k=0,\ldots,n-2$},
\eal
which implies that $Y_n \to Y$ in distribution as $n \to \infty$, where $Y$ is a random variable supported on $(0,1)$ with density $f_Y(x)=2(1-x)$. 

\section{Final Remarks}\label{FinalRemarks}

There are various problems where one can try to apply Stein's method to prove rates of convergence for limit theorems. We mention a few of them here.

\subsection{} Feray \cite{Fer18} introduced the theory of weighted dependency graphs and gave an asymptotic normality criterion in this context. 
As an application, he gave another proof of Flajolet and Noy's central limit theorem for the number of crossings in a random chord diagram. 
It would be interesting to see if there is a Stein's method version for weighted dependency graphs, analogous to the one for regular dependency graphs developed by Baldi and Rinott in \cite{BR89}. 

In another direction, Feray \cite{Fer18} states that weighted dependency graphs can be used to prove a central limit theorem for the number of $k$-crossings. 
It should be possible to use Stein's method to obtain a rate of convergence. 

\subsection{} Chern, Diaconis, Kane, and Rhoades \cite{CDKR15} showed that the number of crossings in a uniformly random set partition of $[n]$ is asymptotically normal. 
Note that a chord diagram is the special case corresponding to a partition of $[2n]$ into $n$ blocks, each of size $2$. 
More recently, Feray \cite{Fer20} generalized this result and used weighted dependency graphs to prove central limit theorems for the number of occurrences of any fixed pattern in multiset permutations and set partitions. 
Is it possible to use Stein's method to obtain rates of convergence?

\subsection{} Fix $n \geq 1$, which will denote the number of vertices in a random graph. Let $\bm{d} = (d_1,\ldots,d_n)$, with $d_i \geq 1$ for all $i \in [n]$, be a degree sequence. 
That is, vertex $i$ has degree $d_i$. Consider $2m$ half-edges (so that two half-edges can be connected to form a single edge), where $2m = \sum_{i=1}^n d_i$, and perform a random matching of these half-edges. 
The resulting random graph $\CM_n(\bm{d})$ is called the {\em configuration model with degree sequence $\bm{d}$} (see \cite{Hof17}, Ch. 7). 
Note that random chord diagrams correspond to the special case where the number of vertices is $2n$ and the degree sequence consists of all $1$'s. 
What is the asymptotic distribution of the number of crossings in the configuration model? 


\section*{Acknowledgements}

JEP thanks Jason Fulman for suggesting the problem of obtaining a rate of convergence for the Flajolet-Noy CLT, introducing him to Stein's method, and providing many useful suggestions and references. 
He also thanks Carina Betken for a helpful conversation about size-bias couplings and introducing him to the configuration model.  
Finally, he thanks Mehdi Talbi, Gin Park, Apoorva Shah, and Greg Terlov for several stimulating discussions. 

SAV and OA thank Larry Goldstein for various communications during the preparation of this paper. OA would like to thank Clemens Huemer for initial discussions on the topic that led to work in this problem.

SAV was supported by a scholarship from CONACYT.
OA was supported by CONACYT Grant CB-2017-2018-A1-S-9764.


\printbibliography[title=References]

\Address

\end{document}